\documentclass[11pt]{article}
\usepackage{amsmath, amsfonts, amssymb, amsthm,epsfig}
\usepackage{setspace}
\usepackage{srcltx}

\let\BFseries\bfseries\def\bfseries{\BFseries\mathversion{bold}} 
\DeclareMathSymbol{\leqslant}{\mathalpha}{AMSa}{"36} 
\DeclareMathSymbol{\geslant}{\mathalpha}{AMSa}{"3E} 
\DeclareMathSymbol{\eset}{\mathalpha}{AMSb}{"3F}     
\renewcommand{\leq}{\;\leqslant\;}                   
\renewcommand{\geq}{\;\geslant\;}                   
\renewcommand{\le}{\;\leqslant\;}                   
\renewcommand{\ge}{\;\geslant\;}                   
\newcommand{\dd}{\,\text{\rm d}}             

\newcommand{\ind}{1\hspace{-0.098cm}\mathrm{l}}

\newcommand{\cH}{\mathcal{H}}
\newcommand{\cI}{\mathcal{I}}

\newcommand{\cX}{\mathcal{X}}

\newcommand{\eps}{\varepsilon}

\newcommand{\sig}{\sigma}

\def\V{\mathbb{V}}                          
\newcommand{\var}{\V}

\newcommand{\IP}{\mathbb{P}}

\newcommand{\IN}{\mathbb{N}_0}

\newcommand{\IR}{\mathbb{R}}
\newcommand{\IE}{\mathbb{E}}

\newcommand{\be}{\begin{eqnarray*}}
\newcommand{\ee}{\end{eqnarray*}}
\newcommand{\ben}{\begin{eqnarray}}
\newcommand{\een}{\end{eqnarray}}

\theoremstyle{plain}
\newtheorem{theo}{Theorem}[section]
\newtheorem{lemma}[theo]{Lemma}
\newtheorem{propo}[theo]{Proposition}
\newtheorem{corollary}[theo]{Corollary}

\theoremstyle{definition}

\newtheorem{remark}[theo]{Remark}
\newtheorem{example}[theo]{Example}

\renewenvironment{proof}[1][] {\smallskip \noindent {\bf Proof#1.} }{\hspace*{\fill}$\square$\medskip\par}
\def\P{{\bf {\mathbb{P}}}}
\newcommand{\pr}[1]{\P(#1)}

\newcommand{\e}{\'{e}}
\def\E{\mathbb{E}}
\newcommand{\norm}[1]{\left\|#1\right\|}
\newcommand{\indi}[1]{\,\ind_{\{#1\}}}
\def\R{\IR}
\def\N{\mathbb{N}_0}

\newcommand{\corr}{\operatorname*{corr}}

\begin{document}
\vglue20pt \centerline{\huge\bf Universality of the asymptotics}
\bigskip

\centerline{\huge\bf  of the one-sided exit problem}
\bigskip

\centerline{\huge\bf  for integrated processes}

\bigskip
\bigskip

\centerline{by}
\bigskip
\medskip

\centerline{{\Large Frank Aurzada and Steffen Dereich}}
\bigskip

\begin{center}\it
Technische Universit\"at Berlin\\
Institut f\"ur Mathematik, MA 7-4 \\
Stra\ss e des 17.\ Juni 136, 10623 Berlin\\
aurzada@math.tu-berlin.de\\
~\\
Philipps-Universit\"at Marburg\\
Fb.\ 12 - Mathematik und Informatik\\
Hans-Meerwein-Stra\ss e, 35032 Marburg\\
dereich@mathematik.uni-marburg.de
\end{center}

\bigskip
\begin{center} \today
\end{center}
\bigskip
\bigskip
\bigskip

{\leftskip=1truecm \rightskip=1truecm \baselineskip=15pt \small

 \noindent{\slshape\bfseries Summary.} We consider the one-sided exit problem for (fractionally) integrated random walks and L\'{e}vy processes. We prove that the rate of decrease of the non-exit probability -- the so-called survival exponent -- is universal in this class of processes. In particular, the survival exponent can be inferred from the (fractionally) integrated Brownian motion.

This, in particular, extends Sinai's result on the survival exponent for the integrated simple random walk to general random walks with some finite exponential moment.

Further, we prove existence and monotonicity of the survival exponent of fractionally integrated processes. We show that this exponent is related to a constant appearing in the study of random polynomials.

\bigskip


\bigskip


\bigskip

\noindent{\slshape\bfseries Running Head.} One-sided exit problem for integrated processes

}

\newpage
\section{Introduction}
\subsection{Statement of the problem}
This article deals with the so-called one-sided exit problem -- also called one-sided barrier problem. For a real-valued stochastic process $(A_t)_{t\ge 0}$ one investigates whether there is a $\theta>0$ such that
\begin{equation} \label{eqn:exit}
\pr{  \sup_{t\in[0,T]} A_t \leq 1  } =  T^{-\theta+o(1)},\qquad \text{as $T\to\infty$.}
\end{equation}
If such an exponent $\theta$ exists it is called the \emph{survival exponent}. The function $F(s)\equiv 1$ acts as a barrier, which the process must not pass. We also discuss different barriers $F$ below.

If the process is self-similar, i.e.\ $(A_{ct})$ and $(c^H A_t)$ have the same finite-dimensional distributions for some $H>0$, the problem is equivalent to the so-called lower tail probability problem:
\begin{equation} \label{eqn:lowertail}
\pr{  \sup_{t\in[0,1]} A_t \leq \eps  } =  \eps^{\theta/H+o(1)},\qquad \text{as $\eps\to0$.}
\end{equation}

Apart from this, we also look at the discrete version of (\ref{eqn:exit}):
\begin{equation} \label{eqn:exitdiscrete}
\pr{  \sup_{n=1,\ldots, N} A_n \leq 1  } =  N^{-\theta+o(1)},\qquad \text{as $N\to\infty$,}
\end{equation}
where $(A_n)_{n\in\IN}$ is a discrete time random process.

Obviously, problems (\ref{eqn:exit}), (\ref{eqn:lowertail}), and (\ref{eqn:exitdiscrete}) are classical questions. They are relevant in a number of quite different applications (see Section~\ref{sec:related}). It is therefore surprising that very little seems to be known about them.

In fact, for (\ref{eqn:exit}) the exponent is known in the following cases: Brownian motion ($\theta=1/2$, trivially obtained via the reflection principle), integrated Brownian motion ($\theta=1/4$, \cite{mckean1963,goldman1971,lachal1991,sinai1991}, see also \cite{groeneboometal1999,lachal1992}), fractional Brownian motion ($\theta=1-H$, \cite{molchan1999,molchan1999b,lishao2004aop}), and some L\e vy processes (see e.g.\ \cite{baxterdonsker,bingham1973}).

It is even more surprising that for the discrete version (\ref{eqn:exitdiscrete}) yet less seems to be known. The only cases where the exponent was calculated are general random walks (e.g.\ $\theta=1/2$ if the increments are centered, see for example \cite{F}) and the integrated \emph{simple} random walk ($\theta=1/4$, \cite{sinai1991}). Bounds for general integrated random walks are given in \cite{caravennadeuschel}, polynomial bounds for integrated Gaussian random walks can be obtained from \cite{lishao2004aop}. In several further special cases, Vysotsky \cite{vlad} obtained $\theta=1/4$. It was conjectured (\cite{caravennadeuschel}) that for any integrated random walk with finite variance the exponent is $\theta=1/4$.

The focus of the present article is on integrated processes and integrated random walks: i.e.\ $A=\cI(X)$ with $\cI$ some integration operator and $X$ a L\'{e}vy martingale or centered random walk. The main motivation for this work was that the exponent was known for integrated Brownian motion, but not for general integrated random walks. We will show that indeed the exponent is $\theta=1/4$ under mild assumptions. We stress that the processes that we considered are non-Markovian.

\subsection{Main results} \label{sec:result}
The goal of this article is to investigate the asymptotics of
\begin{equation}  \label{eqn:todo}
\pr{  \sup_{t\in J\cap [0,T]} A_t \leq 1  }, \qquad \text{as $T\to\infty$,}
\end{equation}
for $J=\N$ or $J=[0,\infty)$. We show the following:
\begin{itemize}
 \item For a fixed integration operator $\cI$, the asymptotics of this probability for $A=\cI(X)$ is universal over the class of L\'{e}vy processes and random walks $X$. The reason for this is that all of these processes can be coupled with a suitable Brownian motion. The resulting order of (\ref{eqn:todo}) can then be inferred from Brownian motion or any other process in this class, such as the simple random walk.
 \item The existence of the survial exponent can be established for the case of fractionally integrated processes. This survival exponent is shown to be decreasing with respect to higher integration. As a byproduct we show that the survival exponent of fractionally integrated Brownian motion (also called Riemann-Liouville process) is not the same as for the corresponding fractional Brownian motion (FBM).
 \item We show a certain robustness concerning the change of the barrier, which is equivalent to adding a drift to the process. In fact, adding a drift to a Gaussian process that is in its reproducing kernel Hilbert space, does not change the survival exponent of that process.
 \item We exploit the connection of the one-sided exit problem to random polynomials established in \cite{demboetal,lishao2004aop} in order to improve the knowledge of the crucial constant appearing there.
\end{itemize}


\bigskip
Let us be more precise. We let $\cX$ denote the class of all (non-deterministic, right-continuous) martingales $(X_t)_{t\geq 0}$ with independent and stationary increments, $X_0=0$, satisfying
$$
\IE [ e^{\beta |X_1|} ] <\infty,\qquad \text{for some $\beta>0$.}
$$
If the martingale is only defined on $\IN$, we set $X_t:=X_{\lfloor t\rfloor}$ for all $t\geq 0$.

Let us further specify the type of functionals $\cI$ that we consider. We let $\cI$ be a functional of the following convolution type:
$$
\cI(X)_t = \int_0^t K(t-s) X_s \dd s,\qquad t\ge 0,
$$
where $K : [0,\infty) \to [0,\infty)$ is a measurable function  satisfying
\begin{align} \label{eqn:assI1}
K(s) \le k [s^{\alpha-1}+s^{\beta-1}]
\end{align}
for positive constants $k$, $\alpha$, and $\beta$ with $\alpha\ge\beta$.
Additionally, we need to impose a regularity assumption on the tail behavior of $K$. Here, we assume that either $K$ is weakly equivalent to a regularly varying function or it is assumed to be ultimately decreasing. We remark that this technical assumption can be further relaxed.

The main example is the integration operator:
$$
\cI_1(X)_t := \int_0^t X_s \dd s, \qquad t\ge 0,
$$
where $K(s)\equiv 1$, but our definition also includes fractional integration operators
\begin{equation} \label{eqn:fracintop}
\cI_\alpha(X)_t := \int_0^t \frac{1}{\Gamma(\alpha)} \, (t-s)^{\alpha-1} X_s \dd s,\qquad t\ge 0,
\end{equation}
where $\alpha>0$ and $\Gamma$ denotes Euler's Gamma function. In particular, if $\alpha$ is an integer, $\cI_\alpha(X)$ is the $\alpha$-times integrated process. For completeness we set $\cI_0$ to be the identity; and we recall that $\cI_\alpha \circ \cI_\beta = \cI_{\alpha+\beta}$ for $\alpha,\beta\geq 0$.

Finally, in order to formulate the main result we call functions $f,g:I\to\IR$ weakly-$\log$-equivalent, if there exists a $\delta>0$ such that
$$
(\log T)^{-\delta} g(T) \precsim f(T)\precsim (\log T)^\delta  g(T).
$$
In that case we briefly write $f\asymp_{\log} g$. Note that $\asymp_{\log}$ defines an equivalence relation. Here and below we use $f \precsim g$ (or $g\succsim f$) if $\limsup f/g < \infty$ and $f\approx g$ if $f\precsim g$ and $g\precsim f$. Further, $f\lesssim g$ (or $g\gtrsim f$) means $\limsup f/g \leq 1$, and $f \sim g$ means that $f\lesssim g$ and $g\lesssim f$.

\bigskip
Using this notation, our main theorem reads as follows.

\begin{theo}\label{thm:main}
Let $(X_t)_{t\geq 0}$ and $(Y_t)_{t\geq 0}$ be two processes from the class $\cX$. Then, for either $J=\N$ or $J=[0,\infty)$,
$$
\pr{ \sup_{t\in J\cap[0,T]} \cI(X)_t \leq 1} \asymp_{\log}  \pr{ \sup_{t\in J\cap [0,T]} \cI(Y)_t \leq 1}.
$$
That means the asymptotics of all processes in the class $\cX$ are equivalent with respect to $\asymp_{\log}$.
\end{theo}

This is the mentioned universality result. In particular, the survival exponent (if it exists) is universal over the class $\cX$. The proof of Theorem~\ref{thm:main} is given in Section~\ref{sec:proofmain1}. In fact, we shall prove a more precise result that gives us control on the logarithmic loss between two processes.

\bigskip
A particularly important case is when $\cI$ is the usual integration operator. Then the rate of the survival probability is known for $X$ being Brownian motion or $X$ being the simple random walk. Our main result entails the following corollary for general random walks.

\begin{corollary} Let $X_1, X_2,\ldots$ be a random walk started in $0$ with $\E[ e^{\beta |X_1|}]<\infty$ for some $\beta>0$ and  with $\E[X_1]=0$. Set $A_n=\sum_{i=1}^n X_i$. Then, as $N\to\infty$,
$$
(\log N)^{-4} N^{-1/4} \precsim \pr{ \sup_{n=1,\ldots,N} A_n \leq 1 } \precsim (\log N)^{4} N^{-1/4}.
$$
\end{corollary}

Similarly, we obtain the result for integrated L\e vy processes.

\begin{corollary} Let $(X_t)_{t\geq 0}$ be a real-valued L\e vy process with $\E[ e^{\beta |X_1|}]<\infty$ for some $\beta>0$ and with $\E[X_1]=0$. Set $A_t:=\int_0^t X_s \dd s$. Then, as $T\to\infty$,
$$
(\log T)^{-4} T^{-1/4} \precsim \pr{ \sup_{t\in[0,T]} A_t \leq 1 } \precsim (\log T)^{4} T^{-1/4}.
$$
\end{corollary}

\bigskip
Theorem~\ref{thm:main} implies that the survival exponent $\theta$ is the same for any process from the class $\cX$. Of course, it is not clear that the exponent exists, but if it does for one process from the class $\cX$ so it does for any process from that class. Now we prove that the survival exponent does indeed exist for the particularly important case of the $\alpha$-fractional integration operator (\ref{eqn:fracintop}) and that it is decreasing in $\alpha$.

\begin{theo} There is a non-increasing function $\theta : [0,\infty) \to (0,1/2]$, $\theta : \alpha\mapsto \theta(\alpha)$, such that for any process $X$ from the class $\cX$ and any $\alpha\geq 0$
$$\label{eqn:defnthetaalpha}
\pr{ \sup_{t\in[0,T]} \cI_\alpha (X)_t \leq 1 } = T^{-\theta(\alpha)+o(1)},\qquad \text{as $T\to\infty$.}
$$
We recall that $\theta(0)=1/2$ and $\theta(1)=1/4$. \label{thm:expfibmdecreasing}
\end{theo}

The proof of Theorem~\ref{thm:expfibmdecreasing} is given in Section~\ref{sec:exisproof}. Theorem~\ref{thm:expfibmdecreasing} does not yield new values for $\theta$, so it remains a challenge to calculate $\theta(\alpha)$, e.g.\ for integers $\alpha$. A lower bound for $\theta$ is obtained in Corollary~\ref{cor:slep}.

\bigskip The connection to random polynomials is discussed in Section~\ref{sec:randpoly}. We give some further remarks in Section~\ref{sec:sammels1}. In Section~\ref{sec:related}, we comment on some related work. The proof of the main result, Theorem~\ref{thm:main}, is given in Section~\ref{sec:proofmain1}. Sections~\ref{sec:apriori} and~\ref{sec:fkg} may be of independent interest: the former contains an a priori estimate for $\cI$ being the identity; the latter a version of the FKG inequality for processes with independent increments. In Section~\ref{sec:drift}, we show that adding a drift of a certain strength to the process $X$ does not influence the survival exponent. This newly developed drift argument also allows to change the barrier $F$. Using this drift argument, the proof of Theorem~\ref{thm:expfibmdecreasing} is an easy consequence.

\subsection{Random polynomials having few or no real zeros} \label{sec:randpoly}
Let us now give an application of our results to the study of zeros of random polynomials. The connection to the one-sided exit problem was established in \cite{demboetal,lishao2004aop}. It was shown in \cite{demboetal} that for $\xi_i$ i.i.d.\ Gaussian random variables
$$
\pr{ \sum_{i=0}^{2n} \xi_i x^i \leq 0~~\forall x\in\R } =n^{-b+o(1)}, \qquad n\to\infty,
$$
where
$$
b:=- 4 \lim_{T\to \infty} \frac{1}{T}\, \log \pr{ \sup_{t\in [0,T]} Y^ \infty_t \leq 0}
$$
and $Y^\infty$ is the stationary Gaussian process with correlation function
$$
{\corr}_\infty(\tau) := \E [Y^\infty_0 Y^\infty_\tau] = \frac{2 e^{-\tau/2}}{1+e^{-\tau}}.
$$
It was shown that $0.4<b<1.29$ (see \cite{demboetal,lishao2004aop}). Here we show the following connection to our problem and an improvement for the numerical value of $b$.

\begin{corollary} \label{cor:slep}
For the decreasing function $\theta$ defined in Theorem~\ref{thm:expfibmdecreasing} we have
$$\theta(\alpha)\geq b/4,\qquad \text{for all $\alpha\geq 0$}.$$
In particular, $b\leq 4 \cdot \theta(1)=1$.
\end{corollary}

This fact gives a further motivation to find values for $\theta(\alpha)$, $\alpha\notin \{0,1\}$.

\begin{proof} Note that it is sufficient to show the lemma for integer $\alpha$, since $\theta$ is decreasing. Consider the Lamperti transforms of the processes $R^n:=\cI_n(W)$, where $W$ is a Brownian motion, normalized by the square root of its variance:
$$
Y^n_t := n! \sqrt{ 2n+1} \,e^{-(n+1/2) t} R^n_{e^t}.
$$
This is a stationary Gaussian process. One can calculate its correlation function ($\tau\geq 0$):

\begin{align*}
{\corr}_n(\tau) &:= \E [Y^n_0 Y^n_\tau] =   n!^2 (2n+1)  e^{-(n+1/2) \tau} \E [ R^n_{1} R^n_{e^\tau} ]\\
 & =   (2n+1)  e^{-(n+1/2) \tau} \int_0^1 (e^\tau-u)^n (1-u)^n \dd u.
\end{align*}

It is elementary to see that
$$
(2n+1)  e^{-(n+1/2) \tau} \int_0^1 (e^\tau-u)^n (1-u)^n \dd u \leq \frac{2 e^{-\tau/2}}{1+e^{-\tau}},\qquad \tau\geq 0,n\geq 1,
$$
with equality at $\tau=0$. Indeed, note that
$$
e^{-n\tau}\int_0^1 (e^\tau-u)^n (1-u)^n \dd u = \int_0^1 [\sqrt{(1-e^{-\tau} u) (1-u)}]^{2n} \dd u \leq \int_0^1 \left( \frac{1-e^{-\tau} u + 1-u}{2}\right)^{2n} \dd u.
$$
Integrating the latter expression gives
$$
\frac{1}{2n+1} \,\frac{2}{e^{-\tau}+1}\,\left( 1 - \left(1-\frac{e^{-\tau}+1}{2}\right)\right)^{2n+1}\leq \frac{2}{e^{-\tau}+1}.
$$

This implies that, for all $n\geq 1$,
$${\corr}_n(0)={\corr}_\infty(0),\qquad\text{and}\qquad {\corr}_n(\tau) \leq {\corr}_\infty(\tau), \tau\geq 0.$$
Therefore, by Slepian's lemma,
\begin{align*}
b &=- 4 \lim_{T\to \infty} \frac{1}{T}\, \log \pr{ \sup_{t\in [0,T]} Y^ \infty_t \leq 0}\\
 & \leq - 4 \lim_{T\to \infty} \frac{1}{T}\, \log \pr{ \sup_{t\in [0,T]} Y^n_t \leq 0}\\
 &= - 4 \lim_{T\to \infty} \frac{1}{T}\, \log \pr{ \sup_{t\in [0,T]} R^n_{e^t} \leq 0}\\
 &= - 4 \lim_{T\to \infty} \frac{1}{T}\, \log \pr{ \sup_{t\in [1,e^T]} R^n_{t} \leq 0}\\ 
 &= - 4 \lim_{T\to \infty} \frac{1}{\log T}\, \log \pr{ \sup_{t\in [1,T]} R^n_{t} \leq 0}\\
 &= 4 \cdot \theta(n),
\end{align*}
where the last step follows from Corollary~\ref{cor:drift3}.
\end{proof}

\subsection{Further remarks} \label{sec:sammels1}
Let us consider fractional Brownian motion (FBM) with Hurst parameter $H\in(0,1)$. It is a close relative of the $\alpha$-fractionally integrated Brownian motion (also called Riemann-Liouville process) with $\alpha:=H-1/2>0$ defined by:
\begin{equation}
R_t^\alpha := \cI_\alpha(W)_t = \frac{1}{\Gamma(\alpha)} \, \int_0^t (t-s)^{\alpha-1} W_s \dd s ,\qquad t\geq 0,
\label{eqn:fbmvsfibm}
\end{equation}
where $W$ is a Brownian motion. For completeness we set $R^0:=W$. Let furthermore
$$
X^\alpha:=R^\alpha +  M^\alpha, \quad\text{where}\quad M_t^\alpha:=\frac{1}{\Gamma(H-1/2)} \int_{-\infty}^0 \left( (t-s)^{H-1/2}-(-s)^{H-1/2}\right)\dd W_s.
$$
Then $X^\alpha$ is a fractional Brownian motion with Hurst parameter $H=\alpha+1/2$.

For $\alpha$-fractionally integrated Brownian motion the survial exponent is given in Theorem~\ref{thm:expfibmdecreasing}. Further, we recall that the survival exponent for FBM with Hurst parameter $H$ is known to be $\theta_{\rm FBM}=1-H$, see \cite{molchan1999}. In view of Theorem~\ref{thm:expfibmdecreasing} (the function $\theta$ is decreasing and $\theta(1)=1/4$ for the $\alpha$-fractionally integrated Brownian motion), it is clear that the survival exponents of both processes cannot coincide. This fact may come as a surprise since often properties of $X^\alpha$ are the same as those of $R^\alpha$.

\begin{corollary} For $\alpha\in(1/4,1/2)$, the survival exponent of $\alpha$-fractionally integrated Brownian motion $R^\alpha=\cI_\alpha(W)$ is not equal to the survival exponent of FBM with the corresponding Hurst parameter $H:=\alpha+1/2$.
\end{corollary}


\bigskip
As a last remark, note that our main theorem considers the behavior of
$$
\pr{ \sup_{t\in J\cap [0,T]} \cI(X)_t \leq 1}
$$
where in the supremum either $J=[0,\infty)$ or $J=\N$. The question arises,  whether is it true that, for any process $X$ from the class $\cX$,
$$
\pr{ \sup_{t\in[0,T]} \cI(X)_t \leq 1} \approx \pr{ \sup_{t\in\N\cap [0,T]} \cI(X)_t \leq 1}.
$$

One can answer this question affirmatively for the case that $\cI$ is the usual integration and $X$ is a discrete process, since $(\cI(X)_t)_{t\geq 0}$ is the linear interpolation of $(\cI(X)_t)_{t\in \N}$, which gives:
$$
\pr{ \sup_{t\in\N\cap [0,\lceil T\rceil ]} \cI(X)_t \leq 1}\leq \pr{ \sup_{t\in[0,T]} \cI(X)_t \leq 1} \leq \pr{ \sup_{t\in\N\cap [0,T]} \cI(X)_t \leq 1},\qquad T>0.
$$
We conjecture that it also holds under suitable conditions on $\cI$.

\subsection{Related work} \label{sec:related}
Let us comment on some further related work and the relevance of the questions (\ref{eqn:exit}), (\ref{eqn:lowertail}), and (\ref{eqn:exitdiscrete}) for other problems.

Li and Shao \cite{lishao2004aop,lishao2005cosmos} are the first who aim at building a theory for a whole class of processes. In the mentioned works, the lower tail probability problem (\ref{eqn:lowertail}) is studied for Gaussian processes. It is shown that the decrease in (\ref{eqn:lowertail}) is indeed on the polynomial scale for many one-dimensional Gaussian processes. However, the technique does not yield values for the survival exponent. An important tool in the study of the above problems for Gaussian processes is the Slepian lemma \cite{slepian1962} and a comparable opposite inequality from \cite{lishoacomparison2002}.

The survival exponent is unknown for the integrated fractional Brownian motion, see \cite{molchankhokhlov2004}. A related question for the Brownian sheet is solved in \cite{csaskiatal2000,csaskiatal1999}. Further references with partial results are \cite{sinai1997,marcus1999,basseisenbaumshi2000}.

We further mention a recent work of Simon \cite{simon2007}, where the problem is studied for certain integrated stable L\e vy processes (in particular, with heavy tails). Even though we also study integrated L\e vy processes in this paper, the results and techniques are completely disjoint.

We finally mention that the survival exponent has a deeper meaning in several models, in particular, in statistical physics when studying the fractal nature of the solution of Burgers' equation, see \cite{sinai1991,frisch,bertoin-burg,molchan1999,molchan1999b,molchan2000,simonburgers}. Apart from this, the exponent plays a role in connection with pursuit problems (see \cite{lishao2004aop} and references therein), in the study of most visited sites of a process (see e.g.\ \cite{basseisenbaumshi2000}), and in the investigation of zeros of random polynomials (see \cite{demboetal} and references therein and Section~\ref{sec:randpoly} above). We refer to \cite{lishao2004aop} for a recent overview of the applications. The question can also be encountered in the physics literature, see \cite{maju} for a summary. The discrete version (\ref{eqn:exitdiscrete}) is studied in connection with random polymers, see \cite{caravennadeuschel}.

\section{Proof of the universality result}
\subsection{A priori estimate via Skorokhod embedding} \label{sec:apriori}
In the proof of Theorem~\ref{thm:main}, we need an a priori estimate for  $X$ from the class $\cX$ of the form
\begin{align}\label{eq1111-3}
\IP(\sup_{t\in[0,T]} X_t\leq \frac{1}{T^{\alpha}})\ge c\, T^{-\delta},\qquad T\ge T_0,
\end{align}
for some $\delta>0$ and $T_0>0$. Here we provide a way of obtaining such an estimate. We do not require finite exponential moments in this context.

\begin{propo} \label{prop:skor} Let $X$ be  either a L\'{e}vy martingale or a random walk with centered increments with $\var(X_1)=\sig^2>0$. Let $(b_t)_{t\geq 0}$ be such that $b_t^2/t\to0$ and $b_t\succsim  t^{-\delta}$, as $t\to\infty$, for some $\delta\ge 0$. Suppose that $\E |X_1|^{2p}<\infty$ for some $p>2\delta+1$. Then we have
$$
\IP(\sup_{s\in [0,t]} X_s\leq b_t) \gtrsim  \sqrt{\frac {2 b_t^2}{\pi \sig^2 t}} \qquad \text{as $t\to\infty$.}
$$
\end{propo}

\begin{remark}
Note that the estimate is sharp in the sense that one gets $\sim$ instead of $\gtrsim$ if $X$ is a Brownian motion.
\end{remark}

\begin{proof} Fix $p$ such that $p> 2\delta+1$ and $\E |X_1|^{2p}<\infty$, where $\delta$ is as in the statement of the proposition.

\emph{Embedding.} We apply a Monroe~\cite{Mon72} embedding. On an appropriate filtered probability space (possibly one needs to enlarge the underlying probability space), one can define a (right-continuous) family of finite \emph{minimal} stopping times $(\tau(t))_{t\geq 0}$ and a Brownian motion $(W_t)$ such that almost surely
$$
X_t= W_{\tau(t)}
$$
for all times $t\geq 0$. By minimality of $\tau(1)$, we conclude that $(W_{t\wedge \tau(1)})$ is uniformly integrable. Hence,  $\IE[\tau(1)]=\IE[W_{\tau(1)}^2]=\IE[X_1^2]=\sig^2<\infty$. Moreover, by the Burkholder-Davis-Gundy inequality (BDG inequality), one has 
$$
\IE[\tau(1)^p]= \IE \bigl [ [W]_{\tau_1}^{p}\bigr]  \leq c_1 \IE[ \sup_{s\in[0,\tau(1)]} |W_s|^{2p} ],
$$
where $c_1=c_1(p)$ is a constant that depends only on $p$. Since $(W_{t\wedge \tau(1)})$ is a uniformly integrable martingale, we get with Doob's inequality that
\begin{align}\label{eq0827-1}
\IE[\tau(1)^p]  \leq c_2 \,\IE[  |W_{\tau(1)}|^{2p} ]=c_2 \,\IE[  |X_1|^{2p} ]<\infty,
\end{align}
where $c_2=c_2(p)$ is an appropriate constant.

Since $(X_t)$ has stationary and independent increments, the embedding can be established such that $(\tau(t))_{t\geq 0}$ itself has stationary and independent increments, see \cite{Mon72}. Hence,
$(\tau(t)-\sig^2 t)$ is a martingale and we conclude with the BDG inequality that
$$
\IE[(\tau(t)-\sig^2 t)^{p}] \leq \IE [\sup_{s\in[0,\lceil t\rceil] } (\tau(s)-\sig^2 s)^{p}] \leq c_3\, \IE\bigl[  [\tau(\cdot)-\sig^2 \cdot]_{\lceil t\rceil }^{p/2}\bigr],
$$
where $c_3=c_3(p)$ is an appropriate constant. Here, $[\cdot]$ denotes the classical bracket process. Next, we apply the triangle inequality together with the stationarity of $(\tau(t)-\sig^2 t)$ to conclude that
$$
\IE[(\tau(t)-\sig^2 t)^{p} ] \leq c_3\,\lceil t\rceil^{p/2} \,\IE\bigl[  [\tau(\cdot)-\sig^2 \cdot]_{1}^{p/2	}\bigr].
$$
It remains to verify the finiteness of the latter expectation. 
First observe that by the BDG inequality
$$
\IE\bigl[  [\tau(\cdot)-\sig^2 \cdot]_{1}^{p/2}\bigr]\leq c_4\, \IE\bigl[ \sup_{s\in[0,1]}  |\tau(s)-\sig^2 s |^{p}\bigr]\leq 2^{p} c_4\bigl( \IE[ \tau(1)^{p}] +\sig^{2p}\bigr),
$$
where $c_4=c_4(p)$ is an appropriate constant. By~(\ref{eq0827-1}),  $\IE[\tau(1)^{p}]$ is finite, and there exists a constant $c_5$ depending on $p$ and the $2p$-th moment of $X_1$ such that for all $t>0$
\begin{align}\label{eq0818-1}
\IE[(\tau(t)-\sig^2 t)^{p}] \leq c_5\, \lceil t\rceil ^{p/2}.
\end{align}

\emph{Estimate of the probability.} Fix $\eps>0$ and observe that
\begin{align}\label{eq0818-2}
\IP(\sup_{s\in[0,t]} X_s\leq b_t) \ge \IP(\sup_{s\in[0, (1+\eps)t\sig^2]} W_{s}\leq b_t) - \IP(\tau(t)\ge (1+\eps) \sig^2 t).
\end{align}
Note that the first term on the right hand side of the latter equation can be computed explicitly:
\begin{align*}
 \IP(\sup_{s\in[0, (1+\eps)t\sig^2]} W_{s}\leq b_t)= \sqrt{\frac 2\pi} \int_0^{\frac{b_t}{\sqrt{(1+\eps)t \sig^2}}} e^{-\frac{y^2}{2}}\dd y\sim \sqrt{\frac 2\pi} \frac{b_t}{\sqrt{(1+\eps)t \sig^2}}.
\end{align*}
In the last step, we used that $b_t^2/t\to0$. Conversely, the second term in (\ref{eq0818-2}) can be controlled via the Chebyshev inequality and~(\ref{eq0818-1}):
$$
\IP(\tau(t)- \sig^2 t\ge \eps \sig^2 t)\leq \frac{\IE[ |\tau(t)-\sig^2 t|^p]}{ (\eps \sig^2 t)^p}\leq c_5\,\frac{\lceil t\rceil^{p/2}}{(\eps\sig^2 t)^p}\approx t^{-p/2}.
$$
By the choice of $p$, the second term on the right hand side of~(\ref{eq0818-2}) is of lower order than the first term. We obtain the lower bound in the proposition by letting $\eps$ tend to zero.
\end{proof}


\begin{remark}[Polynomial behavior of the survival probability]\label{noexpdecay}
Let $X$ be  either a L\'{e}vy martingale or a random walk with centered increments. Let $\alpha$ be as in (\ref{eqn:assI1}) and note that there exists a constant $c\in(0,\infty)$ such that $\int_0^t K(s)\dd s \le c\, t^\alpha$ for all $t\ge 1$. We conclude that for $T\ge 1$
$$
\sup_{t\in [0,T]} \cI(X)_t \le c \,T^\alpha\sup_{t\in [0,T]} X_t ;
$$
so that, by Proposition~\ref{prop:skor},
$$
\IP(\sup_{t\in [0,T]} \cI(X)_t\le 1) \ge  \IP(\sup_{t\in [0,T]} X_t\le (c \,T^\alpha)^{-1}) \succsim T^{-(\alpha+\frac12)},
$$
if  $\E |X_1|^{2p}$ is finite for some $p>2\alpha+1$. In particular, the survival probability cannot decay faster than polynomially in our general setting.
\end{remark}

The estimate from the previous remark is far from optimal in general. We shall use it as an a priori estimate.

Finally, we also recall the following result for Brownian motion with drift. It can be obtained from the distribution of the first hitting time of Brownian motion with a line, which is explicitly known, see e.g.\ \cite{steele}, p.\ 217.
\begin{lemma}\label{lem:apbm} Let $\sigma>0$ and $W$ be a Brownian motion. Then
$$
\pr{ \sigma W_t \leq 1-\frac{t}{\sqrt{T}}, \forall t\leq T} \succsim T^{-1/2}.
$$
\end{lemma}

\subsection{The FKG inequality} \label{sec:fkg}
We will use a version of the FKG inequality for processes with independent increments. Even though the proof follows the standard method (see e.g.\ \cite{grimmett}), we were unable to find this result in the literature.

We call a function $f : \R^n \to \R$ increasing (decreasing, respectively) if for any vectors $x=(x_1,\ldots,x_n)$ and $y=(y_1,\ldots,y_n)$ with $x_1\geq y_1$ and $(x_i-y_i)_{i=1}^n$ increasing we have $f(x)\geq f(y)$ ($f(x)\leq f(y)$, respectively). Note that if $f$ is increasing (decreasing) in each component then it is increasing (decreasing) in this sense.


\begin{theo} \label{thm:fkg}
Let $(X_t)_{t\geq 0}$ be a stochastic process with independent increments. Fix $n\in \N$ and let $f, g : \R^n \to \R$ be measurable functions that are either increasing or decreasing (in the sense defined above).  Then, for any choice of $0\leq t_1\leq \ldots\leq t_n$ such that $\E [ |f(X_{t_1}, \ldots, X_{t_n})\wedge 0| ]< \infty$ and $\E [ |g(X_{t_1}, \ldots, X_{t_n})\wedge 0|] < \infty$, we have
$$
\E [ f(X_{t_1}, \ldots, X_{t_n} ) g(X_{t_1}, \ldots, X_{t_n}) ] \geq \E [f(X_{t_1}, \ldots, X_{t_n})] \,\E [g(X_{t_1}, \ldots, X_{t_n})].
$$
\end{theo}


\begin{proof}
Set
$$
\tilde{f}(x_1,\ldots,x_n):=f(x_1,x_1+x_2, \ldots, x_1+\ldots+x_n).
$$
Since $f$ is increasing in the sense defined above, $\tilde{f}$ is increasing in each component. Analogously, we define $\tilde{g}$. Note that
$$
f(X_{t_1},\ldots,X_{t_n})= \tilde{f}(X_{t_1},X_{t_2}-X_{t_1},\ldots, X_{t_{n}}-X_{t_{n-1}}).
$$
Due to this observation and the fact that $(X_{t_1},X_{t_2}-X_{t_1},\ldots, X_{t_{n}}-X_{t_{n-1}})$ is a vector with independent components, the usual FKG inequality for the product measure (as it can be proved using the technique in e.g.\ \cite{grimmett}) gives us the assertion.
\end{proof}

\subsection{Proof of Theorem~\ref{thm:main}} \label{sec:proofmain1}
Here we give the proof of Theorem~\ref{thm:main}. In fact, we shall prove the following more precise result that gives us control on the logarithmic loss. Theorem~\ref{thm:main} immediately follows from it.

\begin{theo} \label{theo:logestimates}
Let $X$ be a process from the class $\cX$, $W$ be a Brownian motion, and $\cI$ a functional as specified above. Then we have
\begin{equation} \label{eqn:logterms}
(\log T)^{-2(1+\alpha)} \precsim  \frac{\pr{ \sup_{t\in [0,T]} \cI(X)_t \leq 1 }}{\pr{ \sup_{t\in [0,T]} \cI(W)_t \leq 1 }} \precsim (\log T)^{2(1+\alpha)},\qquad\text{as $T\to\infty$.}
\end{equation}
\end{theo}

This suggests the question up to which order the two expressions can differ over the class $\cX$.

\begin{proof}[ of Theorem~\ref{theo:logestimates}]
For an arbitrary fixed process  $X$ from the class $\cX$ and a Wiener process $W$, we shall show that
\begin{equation} \label{eqn:as1a}
\IP(\sup_{t\in J\cap [0,T]} \cI(X)_t\leq 1) \ge c\, (\log T)^{- 2(\alpha+1)}\, \IP(\sup_{t\in J\cap [0,T]} \cI(W)_t \leq 1),
\end{equation}
for $T$ large enough and some constant $c>0$. The opposite bound follows by the same method when exchanging  the roles of $W$ and $X$.\smallskip

\emph{Step 1:} In the first step, we derive one of the key techniques used in the proof (an appropriate coupling of $X$ and $\sig W$ with $\sig>0$ and $\sig^2=\var(X_1)$) from the Koml\'os-Major-Tusn\'ady coupling. 
Since $X_1$ has finite exponential moments in a neighborhood of zero, one can couple the process $X$ with $\sig W$, by the KMT theorem \cite{kmt}, for each fixed $T\in\IN$ such that 
there exist positive constants $\beta_1,\beta_2$ not depending on $T$ with
\begin{align}\label{eq1111-1}
\IE \bigl[\exp( \beta_1 \sup_{t\in \{0,\dots,T\}} |X_t- \sig W_t| )\bigr] \leq \exp( \beta_2 \log (T\vee e)).
\end{align}
As we indicate next, we can take the supremum in the last equation equally well over the interval $[0,T]$ with $T\in(0,\infty)$ (possibly with different constants $\beta_1,\beta_2$).
If $X$ is a L\'evy martingale, then  we get with Doob's inequality for~$\beta_3>0$ that
$$
\IE\bigl[\exp( \beta_3 \sup_{t\in[0,1]} |X_t|)\bigr]= \IE \Bigl[ \Bigl( \sup_{t\in[0,1]}e^{ \frac{\beta_3}2  |X_t|}\Bigr)^2\Bigr]\le 4 \, \IE\bigl[ e^{ \beta_3 |X_1|}\bigr].
$$
Consequently,
$$
\IE \Bigl[ \sup_{t\in\{1,\dots,T\}} \exp( \beta_3 \sup_{s\in[t-1,t]} |X_s-X_{t-1}|)\Bigr] \le T\, \IE\bigl[\exp ( \beta_3 \sup_{t\in[0,1]} |X_t|)\bigr] \leq 4  T \, \IE\bigl[ e^{ \beta_3 |X_1|}];
$$
and the right hand side is finite as long as $\beta_3$ is sufficiently small. One gets an analogous estimate when replacing the L\'evy process by the Wiener process. Now, an application of the triangle inequality together with straightforward calculations yield the mentioned stronger version of~(\ref{eq1111-1}).

We fix $\rho>\alpha+1/2$. By the exponential Chebyshev inequality, we get for $T\ge e$ and arbitrary $a>0$
$$
\IP(\sup_{t\in [0,T]} |X_t- \sig W_t| \ge a) \leq e^{-\beta_1 a} \,T^{\beta_2}
$$
which implies  for $a_T:=\frac{\beta_2 +\rho}{\beta_1} \log T$ that
\begin{align}\label{eq0801-1}
\IP\bigl(\sup_{t\in  [0,T]} |X_t-\sig W_t|\ge a_T \bigr) \leq T^{-\rho}.
\end{align}

\emph{Step 2:} In order to prove (\ref{eqn:as1a}), we consider a particular scenario for which  $\sup_{t\in[0,T]} \cI(X)_t\le1$ is satisfied. We couple $X$ and $\sig W$ on the time interval $[0,T_0]$ as described above. Moreover, we apply the same coupling for the two processes $(X_t-X_{T_0})_{t\in [T_0,T]}$ and $(\sig W_t-\sig W_{T_0})_{t\in[T_0,T]}$. Certainly, both couplings can be established on a common probability space in such a way that the random variables involved in the first coupling are independent from the ones involved in the second coupling.

We fix $\delta_1,\delta_2>0$ with $\delta_2 \int_0^{\delta_1} K(s)\dd s\ge 1$
and consider the barriers
$$
\bar g_T(t):=1-\frac t{\sqrt{T_0}} +a_T \ \text{ and } \  g_T(t):=1-\frac t{\sqrt{T_0}},
$$
where  $T_0=T_0(T)= \lceil \left(2a_T+\delta_2\sig+1\right)^2 \rceil$. Then $\bar g_T(T_0)\le -a_T-\delta_2\sig$.

As we will show next, for any sufficiently large $T$, the event  $\{\sup_{t\in J\cap [0,T]} \cI(X)_t\le 1\}$ occurs at least if all of the following events occur:
$$
E_1=\{X\le \bar g_T \text{ on }[0,T_0]\}, \qquad E_2=\{\sup_{t\in[0,T_0]} X_t\le  c_1\, T_0^{-\alpha} \}, 
$$
$$
E_3=\{\sup_{t\in J \cap[0, T-T_0]} \cI(W_{\cdot+T_0}-W_{T_0}) \le 1\}, \ \text{ and } \ E_4=\{\sup_{t\in[T_0,T]} |X_t-X_{T_0}-\sig (W_t-W_{T_0})|\le a_T\},
$$
where $c_1>0$ is a finite constant  with $\int_0^t K(s)\dd s \le c_1 t^\alpha$ for all $t\ge 1$.
Indeed, $E_1$ and $E_2$ imply (together with the regularity assumption on $K$) that
\begin{align}\label{eq1117-1}
\int_0^{T_0-\delta_1} K(t-s)X_s \dd s \le 0 \ \text{ for } t\ge T_0 \ \text{ and } \ X\le -\delta_2\sig \text{  on }[T_0-\delta_1,T_0],
\end{align}
 as long as $T$ (or equivalently $T_0$) is sufficiently large. Moreover,  given that also  $E_4$ occurs, one has for $t\in[T_0,T]$, 
$$X_t\le X_{T_0}+\sig(W_t-W_{T_0})+a_T\le -\delta_2 \sig+ \sig( W_t-W_{T_0}),
$$
so that
$$
\int_ {T_0}^t  K(t-s) X_s \dd s \le \sig \int_{T_0}^t K(t-s) \bigl[W_s-W_{T_0}-\delta_2 \bigr]  \dd s.
 $$
Assuming additionally $E_3$, we conclude with  (\ref{eq1117-1}) that, for all $t\in J\cap [T_0,T]$,
\begin{align*}
\cI(X)_t & = \int_0^{T_0-\delta_1} K(t-s) X_s \dd s + \int_{T_0-\delta_1}^{T_0} K(t-s) X_s \dd s + \int_{T_0}^t  K(t-s) X_s \dd s\\
         & \leq 0 + \int_{T_0-\delta_1}^{T_0} K(t-s) (-\delta_2 \sigma) \dd s + \sigma \int_{T_0}^t  K(t-s)  \bigl[W_s-W_{T_0}-\delta_2 \bigr]  \dd s\\
         & = - \sigma \delta_2\int_{T_0-\delta_1}^{t} K(t-s) \dd s + \sigma \int_{0}^{t-T_0}  K(t-T_0-s)  \bigl[W_{s+T_0}-W_{T_0} \bigr]  \dd s\\
         & \leq  - \sigma \delta_2 \int_0^{\delta_1} K(s) \dd s + \sigma \cdot 1  \\
         & \leq -\sigma+\sigma\leq 1,
\end{align*}
 as long as $T$ is sufficiently large. Note that $\cI(X)_t \leq 1$ also holds on $J\cap [0,T_0]$ due to $E_2$, see Remark~\ref{noexpdecay}.\smallskip

\emph{Step 3:} It remains to estimate the probability of $E_1\cap\dots\cap E_4$. First we estimate $\IP(E_1\cap E_2)$. 
Note that $\ind_{E_1}$ and $\ind_{E_2}$ can both be written as limits of decreasing functions in the sense of Section~\ref{sec:fkg}. Hence, by Theorem~\ref{thm:fkg}, we have 
$\IP(E_1\cap E_2) \ge \IP(E_1)\cdot\IP(E_2)$.  By Remark~\ref{noexpdecay}, we have $\IP(E_2)\succsim T_0^{-\alpha-1/2}$. Moreover, the event $E_1$ occurs whenever the events
$$
E_1'=\{\forall t\in[0,T_0]: \sigma W_t \le g_T\} \ \text{ and } \
E_1''=\{\sup_{t\in[0,T_0]} |X_t-\sig W_t|\le a_T\} 
$$
occur; and we thus have
$$
\IP(E_1)\ge \IP(E_1'\cap E_1'')\ge \IP(E_1')-\IP({E_1''}^c).
$$
By Lemma~\ref{lem:apbm} and by inequality (\ref{eq0801-1}), one has $\IP(E_1')\succsim T_0^{-1/2}$ and $\IP({E_1''}^c)\precsim T^{-\rho}_0$, respectively, so that 
$\IP(E_1)\succsim T_0^{-1/2}$. Altogether we thus obtain
\begin{align}\label{eq1117-2}
\IP(E_1\cap E_2)\succsim T_0^{-(\alpha+1)} \approx (\log T)^{-2(\alpha+1)}.
\end{align}
Moreover, $E_3 \cap E_4$ is independent of $E_1\cap E_2$ and 
$$
\IP(E_3\cap E_4) \ge \IP(E_3) -\IP(E_4^c)\succsim \IP(\sup_{t\in J \cap [0,T]} \cI(W)_t \le 1), 
$$
since  $\IP(E_4^c)\le T^{-\rho}$ is of lower order than $\IP(E_3)\succsim T^{-(\alpha+1/2)}$, see Remark~\ref{noexpdecay}.
Combining this with~(\ref{eq1117-2}) finishes the proof.
\end{proof}

\section{Drift and barriers} \label{sec:drift}
\subsection{The influence of a drift on Gaussian processes}
In this section, we study the influence of a drift on the survival exponent. We show that one can safely add a drift of a certain strength without changing the survival exponent. However, the technique can be formulated rather generally in terms of the reproducing kernel Hilbert space of the Gaussian process.

\begin{propo} \label{propo:drift}
Let $X$ be some centered Gaussian process attaining values in the Banach space $E$ with reproducing kernel Hilbert space $\cH$. Denote by $\norm{.}$ the norm in $\cH$. Then, for each $f\in \cH$ and each measurable $S$ such that $\pr{ X \in S }>0$, we have
$$
e^{- \sqrt{ 2 \norm{f}^2 \log (1/\pr{ X \in S }) } - \frac{\norm{f}^2}{2}}\leq \frac{\pr{ X+f \in S }}{\pr{ X \in S }} \leq  e^{ \sqrt{ 2 \norm{f}^2 \log (1/\pr{ X \in S }) } - \frac{\norm{f}^2}{2}}.
$$
\end{propo}

This statement allows to estimate $\pr{ X+f \in S }$ by the respective probability without drift. Of course, we are interested in the set
$$
S:=S_T:=\lbrace (x_t)_{0\leq t\leq T} \, :\, \sup_{t\in[0,T]} \cI(x)_t \leq 1\rbrace,
$$
where $\cI$ is a functional as specified above. If the order of $\pr{ X \in S_T }$, when $T\to\infty$, is polynomial with exponent $\theta$ then, by Proposition~\ref{propo:drift}, the same holds for  $\pr{ X+f \in S_T }$. Below we will give some examples.

\begin{proof}[ of Proposition~\ref{propo:drift}]
Using the notation from \cite{lifshits}, the Cameron-Martin formula says that
\begin{equation} \label{eqn:changeofmeasure}
\pr{ X+f \in S } = \E \left[\indi{X\in S} e^{ \langle z,X\rangle - \frac{\norm{f}^2}{2} } \right].
\end{equation}
where $z$ in the $L_2$-completion of the dual of $E$ is the functional belonging to the admissable shift $f$, see \cite{lifshits}.

{\it Upper bound.} Let $p>1$ and $1/p+1/q=1$. We use the H\"{o}lder inequality in (\ref{eqn:changeofmeasure}) to get
$$
\pr{ X+f \in S } \leq (\E [\indi{X\in S}^p] )^{1/p} (\E [ e^{ q  \langle z,X\rangle} ] )^{1/q} e^{ - \frac{\norm{f}^2}{2}},
$$
Recall that $\langle z, X\rangle$ is a centered Gaussian random variable with variance $\norm{f}^2$. Therefore, we get
\begin{equation} \label{eqn:lh34}
\pr{ X+f \in S } \leq \pr{X\in S}^{1/p} e^{ q\, \frac{\norm{f}^2}{2}  - \frac{\norm{f}^2}{2}}.
\end{equation}
Optimizing in $p$ shows that the best choice is
$$
1/p:=1-\sqrt{ \frac{\norm{f}^2}{2 \log (1/\pr{X\in S})}}<1. 
$$
Plugging this into (\ref{eqn:lh34}) shows the upper bound in the proposition.

{\it Lower bound.} Here we let $p>1$ and use the reverse H\"{o}lder inequality in (\ref{eqn:changeofmeasure}) to get
$$
\pr{ X+f \in S } \geq (\E [\indi{X\in S}^{1/p}] )^{p} (\E [ e^{ -\frac{1}{p-1}  \langle z,X\rangle} ] )^{-(p-1)} e^{ - \frac{\norm{f}^2}{2}},
$$
As above, we can calculate the second expectation, optimize in $p$ to find that the best choice is
$$
p:=1+\sqrt{ \frac{\norm{f}^2}{2 \log (1/\pr{X\in S})}}>1.
$$
Using this shows the lower bound.
\end{proof}

\subsection{Examples} \label{sec:driftvsbarrier}
Our first example is Brownian motion.

\begin{corollary} Let $W$ be a Brownian motion, $\cI$ be a functional as specified above, $f' : [0,\infty)\to\R$ be a measurable function with $\int_0^\infty f'(s)^2\dd s<\infty$, and set $f(t):=\int_0^t f'(s)\dd s$. Let $\theta>0$. Then
$$
\pr{ \sup_{t\in[0,T]} \cI( W)_t \leq 1}=T^{-\theta+o(1)}\quad \text{if and only if}\quad \pr{ \sup_{t\in[0,T]} \cI( W+ f)_t \leq 1} = T^{-\theta+o(1)}.
$$
Also, upper (lower) bounds imply upper (lower) bounds.  \label{cor:drift2}
\end{corollary}

Let us discuss some important examples of drift functions.

\begin{example}
The first example is $f(t)=t/\sqrt{T}$. Then Proposition~\ref{propo:drift} yields
$$
c T^{-1/2}\geq \pr{ \forall t\leq T\,:\, W_t \leq 1 - \frac{t}{\sqrt{T}}} \succsim T^{-1/2} e^{- \sqrt{\log T}} = T^{-1/2+o(1)}.
$$
Actually, a slightly stronger result holds, cf.\ Lemma~\ref{lem:apbm}.
\end{example}

\begin{example}
It is interesting that one can add drift functions up to $|f(t)|\precsim t^\gamma$, $t\to\infty$, with $\gamma<1/2$. Namely, Corollary~\ref{cor:drift2} yields that for any $c\in \R$ and $0\leq \gamma \leq \frac12$:
$$
\pr{ \sup_{t\in[0,T]} ( W_t + c t^\gamma) \leq 1} = T^{-1/2+o(1)}
$$
and
$$
\pr{ \sup_{t\in[0,T]} ( \int_0^t W_s \dd s + c t^{1+\gamma}) \leq 1} = T^{-1/4+o(1)}.
$$
We remark that the latter statement improves Sina\u\i's result \cite{sinai1991} who showed the statement for $\gamma=0$.
\end{example}

As a further example for a Gaussian process, let us consider the $\alpha$-fractionally integrated Brownian motion defined in (\ref{eqn:fbmvsfibm}). Here, one can add drift functions up to $|f(t)|\preceq t^{\gamma}$, $\gamma<H=\alpha+1/2$.

\begin{corollary} Let $R^\alpha=\cI_\alpha(W)$ be an $\alpha$-fractionally integrated Brownian motion, and let $f' : [0,\infty) \to \R$ be a function with $\int_0^\infty f'(s)^2\dd s<\infty$. Let $\theta>0$ and define
$$
g(t):=\frac{1}{\Gamma(\alpha+1)}\,\int_0^t (t-s)^{\alpha} f'(s)\dd s, \qquad t\geq 0.
$$ Then
$$
\pr{ \sup_{t\in[0,T]} R^\alpha_t \leq 1}=T^{-\theta+o(1)}\quad \text{if and only if}\quad \pr{ \sup_{t\in[0,T]}(R^\alpha_t+g(t)) \leq 1} = T^{-\theta+o(1)}.
$$
Also, upper (lower) bounds imply upper (lower) bounds. \label{cor:drift3}
\end{corollary}

We remark that these results are extremely useful when dealing with different `barriers'. As a `barrier' we consider a function $F : [0,\infty)\to (-\infty,\infty]$ and ask when
$$
\pr{ \forall t\leq T : X_t \leq F(t)}, \qquad \text{as $T\to\infty$},
$$
has the same asymptotics as
$$
\pr{ \forall t\leq T : X_t \leq 1}, \qquad \text{as $T\to\infty$}.
$$
Note that e.g.\ Corollary~\ref{cor:drift3} states the following: it is allowed to replace the barrier $F(t):=1-g(t)$ if $g$ is from the reproducing kernel Hilbert space. Estimates can be obtained if one can find $g$ in the reproducing kernel Hilbert space such that $F(t)\geq 1-g(t)$ or $F(t)\leq 1-g(t)$. We demonstrate this method with the following important example.

\begin{example} \label{exa:drifttrick} Let us consider the barrier
$$
F(t):=\begin{cases}
       \infty & 0\leq t< 1,\\
       0      & 1\leq t \leq T,
      \end{cases}
$$
for the process $R^\alpha=\cI_\alpha(W)$ defined in (\ref{eqn:fbmvsfibm}).
\begin{corollary} \label{cor:drifttrick}
 Let $R^\alpha$ be the $\alpha$-fractionally integrated Brownian motion. Then, for a $\theta>0$,
$$
\pr{ \sup_{t\in [1,T]} R^\alpha_t \leq 0} = T^{-\theta+o(1)}
$$
if and only if
$$
\pr{ \sup_{t\in [0,T]} R^\alpha_t \leq 1} = T^{-\theta+o(1)}.
$$
Also, upper (lower) bounds imply upper (lower) bounds.
\end{corollary}

\begin{proof} We can e.g.\ use the function $f':=\ind_{[0,1]} \Gamma(\alpha+1) (\alpha+1)$, for which $\int_0^\infty f'(s)^2\dd s <\infty$. Then $g(t)=\Gamma(\alpha+1) (\alpha+1) \cI_{\alpha+1}(\ind_{[0,1]})_t = t^{\alpha+1}- (t-1)^{\alpha+1}\geq 1$ for all $t\geq 1$ and all $\alpha>0$, i.e.\
$$
F(t)\geq 1-g(t),\qquad \text{for all $t\in[0,T]$};
$$
and thus
$$
\pr{ \forall t\in[0,T] : R^\alpha_t + g(t) \leq 1} \,\leq\, \pr{\forall t\in[0,T] :  R^\alpha_t \leq F(t)} \,=\, \pr{ \sup_{t\in [1,T]} R^\alpha_t \leq 0}.
$$
Corollary~\ref{cor:drift3} therefore implies one bound in the assertion.

The opposite estimate can be obtained via Slepian's lemma (see the version in Corollary~3.12 in \cite{LT}):
$$
\pr{\sup_{t\in[0,T]} R^\alpha_{t} \leq 1} \geq \pr{\sup_{t\in[0,1]} R^\alpha_{t} \leq 1} \, \pr{\sup_{t\in[1,T]} R^\alpha_{t} \leq 1} \geq \pr{\sup_{t\in[0,1]} R^\alpha_{t} \leq 1} \, \pr{\sup_{t\in[1,T]} R^\alpha_{t} \leq 0}.
$$
\end{proof}
\end{example}

\subsection{Proof of Theorem~\ref{thm:expfibmdecreasing}} \label{sec:exisproof}
Here we give the proof of Theorem~\ref{thm:expfibmdecreasing}. Due to Theorem~\ref{thm:main} it is sufficient to consider the question of the survival exponent in the case when $X$ is a Brownian motion. Therefore, we consider $R^\alpha=\cI_\alpha(X)$, where $X$ is a Brownian motion.

\emph{Proof of the existence:} We use the approach from \cite{lishao2004aop} involving the Lamperti transform. However, we employ the new drift argument developed in Example~\ref{exa:drifttrick} rather than calculations involving the Slepian lemma from \cite{lishao2004aop}, which do not seem to be easily transferable to the present situation. Note that the Lamperti transform of $R^\alpha$,
$$
Y_t := e^{-t(\alpha+1/2)} R^\alpha_{e^t},\qquad t\geq 0,
$$
is a continuous, zero mean, stationary Gaussian process with positive correlations $\E [Y_t Y_0]\geq 0$. Therefore, Slepian's lemma and the standard subadditivity argument (see Proposition~3.1 in \cite{lishao2004aop}) show that the following limit exists and equals the supremum:
\begin{equation} \label{eqn:subadd}
\lim_{T\to\infty} \frac{1}{T}\, \log \pr{ \sup_{t\in[0,T]} Y_t \leq 0} = \sup_{T>0} \frac{1}{T} \, \log \pr{ \sup_{t\in[0,T]} Y_t \leq 0}.
\end{equation}
We shall prove that this limit is actually a representation for $\theta(\alpha)$. To see this, note that
$$
\pr{ \sup_{t\in[0,\log T]} Y_t \leq 0} = \pr{ \sup_{t\in[0,\log T]} e^{-t(\alpha+1/2)} R^\alpha_{e^t} \leq 0} = \pr{ \sup_{t\in[0,\log T]} R^\alpha_{e^t} \leq 0} = \pr{ \sup_{t\in[1,T]} R^\alpha_{t} \leq 0}.
$$
Therefore, the limit in (\ref{eqn:subadd}) equals
$$
\lim_{T\to\infty} \frac{1}{T} \,\log \pr{ \sup_{t\in[0,T]} Y_t \leq 0} = \lim_{T\to\infty} \frac{1}{\log T}\, \log \pr{\sup_{t\in[1,T]} R^\alpha_{t} \leq 0}.
$$
It remains to be shown that
$$
 \pr{\sup_{t\in[1,T]} R^\alpha_{t} \leq 0}=T^{-\theta+o(1)} \qquad\text{implies}\qquad \pr{\sup_{t\in[0,T]} R^\alpha_{t} \leq 1}=T^{-\theta+o(1)}.
$$
However, this was already shown in Corollary~\ref{cor:drifttrick}.

\emph{Proof of the monotonicity:} Let $\gamma\geq 0$ and $0<\alpha<1$. We will show that $\theta(\gamma)\geq \theta(\gamma+\alpha)$.

If $R^\gamma_s\leq \ind_{[0,1]}(s)$, for all $s\in[0,T]$, then
$$
R^{\alpha+\gamma}_t = \cI_\alpha(R^{\gamma})_t \leq \begin{cases}
                                                  \frac{1}{\Gamma(\alpha)}\int_0^t \alpha (t-s)^{\alpha-1} \dd s = \frac{t^\alpha}{\alpha \Gamma(\alpha)} \leq \frac{1}{\alpha \Gamma(\alpha)}  & t\leq 1\\
                                                  \frac{1}{\Gamma(\alpha)}\int_0^1 \alpha (t-s)^{\alpha-1} \dd s = \frac{t^\alpha - (t-1)^\alpha}{\alpha\Gamma(\alpha)}\leq  \frac{1}{\alpha\Gamma(\alpha)} & t\geq 1,\\
                                                 \end{cases}
$$
since $\alpha<1$. Therefore, using also the self-similarity (set $\lambda:=(\alpha \Gamma(\alpha))^{1/(\alpha+\gamma+1/2)}$),
\begin{align*}
& \pr{\forall s\leq T :  R^\gamma_s \leq \ind_{[0,1]}(s)} \\
& \leq \pr{\forall s\leq T : R^{\alpha+\gamma}_s \leq \frac{1}{\alpha\Gamma(\alpha)}} = \pr{\forall s\leq \lambda T  : R^{\alpha+\gamma}_s \leq 1} = T^{-\theta(\alpha+\gamma)+o(1)}.
\end{align*}
The left-hand side is treated with Corollary~\ref{cor:drift3} ($f'=(\gamma+1)\Gamma(\gamma+1) \ind_{[0,1]}$) showing that it behaves as $T^{-\theta(\gamma)+o(1)}$. This shows the monotonicity of the survival exponent.


\bibliographystyle{plain}

\end{document}